\documentclass[10pt]{article}

\usepackage{amsmath}
\usepackage{manfnt}
\usepackage{amssymb}
\usepackage{graphicx}
\usepackage{geometry}
\usepackage{amsthm}
\usepackage[dvipsnames]{xcolor}
\usepackage[utf8]{inputenc}
\usepackage{pst-node}
\usepackage[scr=rsfs]{mathalfa}
\usepackage{tikz-cd} 
\usepackage{verbatim}
\usepackage{hyperref}
\usepackage{pdfpages}
\usepackage[backend=biber]{biblatex}

\graphicspath{{graphics}}

\newcommand{\opIm}{\operatorname{Im}}
\newcommand{\opker}{\operatorname{ker}}
\newcommand{\opId}{\operatorname{Id}}
\newcommand{\opRe}{\operatorname{Re}}
\newcommand{\oploc}{\operatorname{loc}}

\newcommand{\oprank}{\operatorname{rank}}

\newcommand{\opsupp}{\operatorname{supp}}

\newtheorem{Theorem}{Theorem}
\newtheorem{Proposition}[Theorem]{Proposition}
\newtheorem{Remark}[Theorem]{Remark}
\newtheorem{Definition}[Theorem]{Definition}
\newtheorem{Lemma}[Theorem]{Lemma}

\newtheorem{Corollary}[Theorem]{Corollary}

\title{Output tracking: an application for irrational SISO transfer functions and a survey for MIMO systems}
\author{Lucas Davron}
\begin{document}
\maketitle
\begin{abstract}
In this paper we consider (MIMO) finite dimensional linear and time-invariant systems and irrational single-input single-output (SISO) transfer functions. In both cases our aim is to describe as best as possible the range of the input-output map $u(\cdot) \mapsto y(\cdot)$. The theory in finite dimension is quite satisfying but rather scarce, the first aim of this paper is to collect the main results in this direction in a comprehensive way. Our second aim is to improve a recent result on the tracking problem for irrational SISO transfer functions, allowing one to completely describe the outputs of such system for inputs $u \in L^2(0,\infty)$. An application is given for a heat equation with polynomial coefficients. 
\end{abstract}
\section{Introduction}
In this paper we consider linear and time-invariant (LTI) systems, of the form
\begin{equation}\label{eq:equations_LTI}
\left\lbrace \begin{array}{rcl}
\dot{x}(t) &=& Ax(t) + Bu(t), \\
x(0) &=& 0, \\
y(t) &=& Cx(t) + Du(t),
\end{array}\right.
\end{equation}
in the two following situations: either \eqref{eq:equations_LTI} is finite-dimensional, or \eqref{eq:equations_LTI} is infinite-dimensional but is single-input and single-output (SISO). In both cases, our aim is to describe (as best as possible) the range of the input-output map $u(\cdot) \mapsto y(\cdot)$, which we refer to as a \textit{tracking problem}. The theory in finite dimension is quite satisfying, culminating with the construction inverse systems, allowing one to re-construct the command $u(\cdot)$ of \eqref{eq:equations_LTI} from its output $y(\cdot)$, using an auxiliary dynamical system. In our opinion this theory is somehow scarce, our first aim is to survey it. Our second aim is to improve a recent result obtained in \cite{davron_lissy} on the tracking problem for irrational SISO transfer functions. As an application we are able to completely characterize the output signals for 
\begin{equation}\label{eq:bessel}
    \left\lbrace \begin{array}{rcl c c}
        z_t(t,x) &=& x^2 z_{xx}(t,x) + xz_x(t,x) + x^2 z(t,x), &t > 0,&  \ell_1 < x < \ell_2,  \\
        z(t,\ell_1) & = & u(t),& t > 0,\\
        z(t,\ell_2) &=& 0,& t > 0, \\
        z(0,x) &=& 0,&& \ell_1 < x < \ell_2,\\
        y(t) &=& z_x(t,\ell_2), & t > 0,
    \end{array}\right.
\end{equation}
where $u \in L^2(0,\infty)$ and $0 < \ell_1 < \ell_2 < \infty$ belong to a certain set. 
\subsection{Finite dimensional systems}
\subsubsection{Problem statement}
Consider a finite dimensional linear and time-invariant (LTI) system, whose dynamics is governed by the equations \eqref{eq:equations_LTI} with 
\begin{equation}\label{eq:dimensions}
    A \in \mathcal{M}_{d}(\mathbb{C}),\quad B \in \mathcal{M}_{d,p}(\mathbb{C}),\quad C \in \mathcal{M}_{q,d}(\mathbb{C}),\quad D \in \mathcal{M}_{q,p}(\mathbb{C}).
\end{equation}
For more on the general theory of LTI systems in finite dimension, see \textit{e.g.} \cite{Coron,trentelman,sontag,chen}. For such systems, a central theme is the controllability of the full state $x(t)$, which is solved by the celebrated Kalman rank condition \cite[\S 1.3]{Coron}. Based one this we begin with the following definition. 
\begin{Definition}
Let $0 < \tau < \infty$. We say that the system \eqref{eq:equations_LTI} is output-controllable (at time $\tau$) when 
\[
\forall y^\tau \in \mathbb{C}^q,\quad \exists u \in C^0([0,\tau];\mathbb{C}^p),\quad y(\tau) = y^\tau. 
\]
\end{Definition}
In the above, we have taken continuous control laws so that the pointwise evaluation of 
\begin{equation}\label{eq:repr_y}
y(t) = \int_0^t Ce^{(t-\sigma)A}Bu(\sigma)d\sigma + Du(t),
\end{equation}
is well-defined. This choice is meaningless: if $\mathcal{U}$ is a dense subset of $C^0([0,\tau];\mathbb{C}^p)$, we have the set equality 
\[
\{ y(\tau) : u \in \mathcal{U} \} = \{ y(\tau) : u \in C^0([0,\tau];\mathbb{C}^p) \}.
\]
The output-controllability problem has a solution very close to the Kalman rank condition, we introduce the controllability matrix 
\[
\mathfrak{C} := \left[ B | AB | ... | A^{d-1}B \right] \in \mathcal{M}_{d,dp}(\mathbb{C}).
\]
\begin{Proposition}{\cite[Theorem III]{kreindler}} The application $u \mapsto y(\tau)$ has range $\opIm [C \mathfrak{C} | D]$.
\end{Proposition}
In particular, the map $\mathbb{F}$ is bounded $L^2(0,\tau;\mathbb{C}^p) \to L^2(0,\tau;\opIm[C \mathfrak{C} | D])$. Without loss of generality we assume that $\opIm[C \mathfrak{C} | D] = \mathbb{C}^q$, which is equivalent to the output-controllability of \eqref{eq:equations_LTI}. For us the \textit{output tracking problem} is to describe the range of the application
\[
\mathbb{F} : u(\cdot) \mapsto y(\cdot), \quad L^2(0,\tau;\mathbb{C}^p) \to L^2(0,\tau;\mathbb{C}^q),
\]
which is more difficult than the above Proposition because the codomain is now an infinite dimensional vector space. As a preliminary consider the following easy fact. 
\begin{Proposition}\label{prop:exact_tracking_EDO}
Let $0 < \tau < \infty$. Then, $\mathbb{F}$ is surjective  $L^2(0,\tau;\mathbb{C}^p) \to L^2(0,\tau;\mathbb{C}^q)$ if and only if $D$ is surjective.
\end{Proposition}
The hypothesis ``$D$ is surjective" is highly stringent: it means that the action of the controller is strong enough to wipe out the action of the dynamics. This calls for a more precise description of the range of $\mathbb{F}$. 
\subsubsection{Main results}
We collect two facts one obtains working with the existing literature. We emphasize that we do not claim any originality, see \S \ref{sec:literature} for an account of the past works. Denote $\mathbf{H}$ the \textit{transfer function} of \eqref{eq:equations_LTI}, defined by
\begin{equation}\label{eq:def_transfer}
    \mathbf{H}(s) := C(s-A)^{-1}B + D.
\end{equation}
The rational matrix $\mathbf{H}$ encodes the behavior of $\mathbb{F}$, as can be seen from the formula 
\[
\forall u \in L^2(0,\infty;\mathbb{C}^p),\quad \hat{y}(s) = \mathbf{H}(s) \hat{u}(s),
\]
for $s \in \mathbb{C}$ with large enough real part and where $\hat{\ }$ is the Laplace transform. Given an output $y$, it is natural to try to generate it using a control $u$ such that $\hat{u}(s) = \mathbf{G}(s)\hat{y}(s)$, where $\mathbf{G}$ is a right-inverse of $\mathbf{H}$. Based on this we have the following, for which we denote 
\[
H^L_{(0)}(0,\tau;\mathbb{C}^n) = \{ f \in H^L(0,\tau;\mathbb{C}^n) : f(0) = ... = f^{(L-1)}(0) = 0 \}.
\]
\begin{Proposition}\label{prop:dense_range} The following assertions are equivalent to each others. 
\begin{enumerate}
    \item The operator $\mathbb{F} : L^2_{\oploc}([0,\infty); \mathbb{C}^p) \to L^2_{\oploc}([0,\infty);\mathbb{C}^q)$ has dense range.
    \item For all $0 < \tau < \infty$, the operator $\mathbb{F} : L^2(0,\tau;\mathbb{C}^p) \to L^2(0,\tau;\mathbb{C}^q)$ has dense range.
    \item For some $0 < \tau < \infty$, the operator $\mathbb{F} : L^2(0,\tau;\mathbb{C}^p) \to L^2(0,\tau;\mathbb{C}^q)$ has dense range.
    \item The transfer function $\mathbf{H}$ is a surjective rational matrix. 
    \item There is $s \in \rho(A)$ such that 
    \[
    \underset{\mathbb{C}}{\oprank} \left( \begin{array}{cc}
A-s & B \\
C & D 
\end{array} \right) = d +q.
\]
    \item The matrix 
    \[
\mathfrak{T} := \left( \begin{array}{ccccccc}
D & CB & CAB & \cdots & CA^{d-1}B & \cdots & CA^{2d-1}B \\
0 & D & CB & \cdots & CA^{d-2}B & \cdots& CA^{2d-2}B \\
\vdots & \vdots & \vdots &  & \vdots &  & \vdots \\
0 & 0 & 0 & \cdots & D & \cdots & CA^{d-1}B
\end{array}\right),
\]
is surjective. 
    \item There is $L \in \mathbb{N}$ such that for all $0 < \tau < \infty$, we have 
    \begin{equation}\label{eq:incl_sobo_finite}
        H^L_{(0)}(0,\tau;\mathbb{C}^q) \subset \mathbb{F} L^2(0,\tau;\mathbb{C}^p),
    \end{equation}
    \item $d \geq q$ and the inclusion \eqref{eq:incl_sobo_finite} holds with $L = \min(d-q + \oprank D+1,d)$, for all $0 < \tau < \infty$.
\end{enumerate}
\end{Proposition}
\begin{Remark}
\begin{itemize}
    \item The last line of $\mathfrak{T}$ is precisely $[C \mathfrak{C} | D]$. Hence, the denseness of the range of $\mathbb{F}$ implies that \eqref{eq:equations_LTI} is output-controllable. 
    \item For a specific system one may obtain that \eqref{eq:incl_sobo_finite} holds with a possibly lower $L$ with the aid of \cite[Theorem 2]{sain_massey}.
\end{itemize}
\end{Remark}
For a general system \eqref{eq:equations_LTI} the description of $\mathcal{Y}(0,\tau) := \mathbb{F} L^2(0,\tau;\mathbb{C}^p)$ is more involved. The only non-trivial situation where one is able to precisely describe the range of $\mathbb{F}$ is for single-output systems, \textit{i.e.} when $q=1$.
\begin{Proposition}\label{prop:tracking_MISO}
Assume that $q=1$ and $0 < \tau < \infty$. The set $\mathcal{Y}(0,\tau)$ satisfies the alternative: 
\begin{itemize}
\item If $D \neq 0$, then $\mathcal{Y}(0,\tau) = L^2(0,\tau)$
\item If $D=0$ and $CA^pB = 0$ for all $p \in [\![0,d-1]\!]$, then $\mathcal{Y}(0,\tau) = \{ 0 \}$. 
\item If $D = 0$ and $\nu \in [\![0,d-1]\!]$ is such that 
\[
CB = ... = CA^{\nu-1}B = 0,\quad CA^\nu B \neq 0,
\]
then $\mathcal{Y}(0,\tau) = H^{(\nu+1)}_{(0)}(0,\tau)$.
\end{itemize}
\end{Proposition}
\subsection{SISO systems}\label{sec:intro_siso}
We now consider $\mathbb{F}$ to be a bounded operator $L^2_{\oploc}[0,\infty) \to L^2_{\oploc}[0,\infty)$ which commutes with forward shifts. Such operators are called \textit{causal}. There exists a number $\alpha \in \mathbb{R}$ and a holomorphic function $\mathbf{H} : \mathbb{C}_\alpha \to \mathbb{C}$ such that 
\[
\forall u \in L^2(0,\infty),\quad \hat{y}(s) = \mathbf{H}(s) \hat{u}(s), \quad \opRe s > \max(0,\alpha),
\]
where $\mathbb{C}_\alpha := \{ \sigma + i \tau : \sigma > \alpha \}$ and $y = \mathbb{F}u$. We keep calling $\mathbf{H}$ the tranfer function of $\mathbb{F}$. It is further found that for all $\alpha < \beta < \infty$, the function $\mathbf{H}$ is bounded on $\mathbb{C}_\beta$, see \cite{weiss_repr_transfer}. We exploit the method developed in \cite{davron_lissy} to characterize the set 
\[
\mathcal{Y}(0,\infty) := \mathbb{F} L^2(0,\infty),
\]
using the complex analytic properties of $\mathbf{H}$. For simplicity we will assume $\mathbf{H}$ is bounded on $\mathbb{C}_0$ and is continuous on $\mathbb{C}_0 \cup i \mathbb{R}$, so that in particular $\mathbb{F}$ is bounded $L^2(0,\infty) \to L^2(0,\infty)$. We will obtain the following. 
\begin{Theorem}\label{theo:charac_output_cartwright}
Assume that $\mathbf{H}$ is as above and never vanishes on $\mathbb{C}_0 \cup i \mathbb{R}$. Then, the elements $y$ of $\mathcal{Y}(0,\infty)$ are precisely those $y \in L^2(0,\infty)$ such that
\begin{equation}\label{eq:charac_output}
\int_\mathbb{R} \left|\frac{\mathcal{F}y(\tau)}{\mathbf{H}(i\tau)}\right|^2 d\tau < \infty,\quad \inf \opsupp y \geq -\liminf_{\sigma \to \infty} \frac{\log|\mathbf{H}(\sigma)|}{\sigma},
\end{equation}
\end{Theorem}
In the above, $\mathcal{F}$ is the Fourier transform
\[
\mathcal{F}f(\xi) = \int_\mathbb{R} e^{-ix\xi} f(x) dx. 
\]
Theorem \ref{theo:charac_output_cartwright} is essentially shown in \cite{davron_lissy} (see the proofs of theorems 1.1 and 3.1 therein) under the additional assumption that $1/\mathbf{H}$ has exponential order $<1$ on $\mathbb{C}_0$. The latter uses \cite[Lemma 3.4]{foures}, from which one deduces Theorem \ref{theo:charac_output_cartwright} under the weaker assumption: $1/\mathbf{H}$ has exponential type $0$ on $\mathbb{C}_0$. Thus, Theorem \ref{theo:charac_output_cartwright} improves the latter by allowing $1/\mathbf{H}$ to have arbitrary (but finite) type. More precisely, we make the delay
\[
\delta := -\liminf_{\sigma \to \infty} \frac{\log|\mathbf{H}(\sigma)|}{\sigma},
\]
sharp. This is done by taking advantage of the fine properties of the Cartwright class. 
\newline
\newline
For the purpose of illustration we replicate the argument of \cite{davron_lissy} to translate the first condition of \eqref{eq:charac_output} in the time variable. The idea is to consider the term $1/|\mathbf{H}(i\tau)|$ as a weight and the square summability requirement in \eqref{eq:charac_output} as a regularity condition. In view of \cite[Theorem 1.7]{davron_lissy}, we immediately deduce the following. 
\begin{Corollary}\label{coro:plancherel_gevrey}
Under the assumptions of the previous Theorem, assume furthermore that there exists $s,R > 0$ and $\gamma \in \mathbb{R}$ such that 
\[
\mathbf{H}(i\xi)^{-1} \asymp (1+|\xi|)^\gamma e^{R |\xi|^{1/s}}, \quad \xi \in \mathbb{R}.
\]
Then, the elements of $\mathcal{Y}(0,\infty)$ are precisely those $y \in C^\infty(\mathbb{R})$ such that 
\begin{equation}\label{eq:output_DC}
    y(t) = 0, \ (t < \delta),\quad \sum_{n=0}^\infty \left( \frac{\|y^{(n)}\|_{L^2(\mathbb{R})}}{M_n} \right)^2 < \infty,\quad M_n := \frac{(ns)!}{R^{ns}} (1+n)^{-s\gamma -1/4}. 
\end{equation}
\end{Corollary}
As an application we characterize the outputs of \eqref{eq:bessel} when the unbounded operator on $L^2(\ell_1,\ell_2)$ given by 
\[
A_0 = -x^2 \partial_{xx} - x \partial_x -x^2 \opId,\quad D(A_0) = H^2(\ell_1,\ell_2) \cap H^1_0(\ell_1,\ell_2), 
\]
is strictly positive (equivalently, has positive lowest eigenvalue). The eigenvalues of $A_0$ are closely linked with the Bessel functions, using the known properties of these functions we obtain the following.
\begin{Theorem}\label{theo:bessel}
Assume that $A_0>0$. Then, the output signals of \eqref{eq:bessel} are these $y \in C^\infty(\mathbb{R})$ such that \eqref{eq:output_DC} holds with 
\[
\delta = 0,\quad R = \frac{1}{\sqrt{2}} \log \frac{\ell_1}{\ell_2},\quad s = 2,\quad \gamma = 0.
\]
\end{Theorem}
\begin{Remark}
\begin{itemize}
    \item Observe that from the Poincaré inequality 
\[
\forall u \in H^1_0(\ell_1,\ell_2),\quad \int_{\ell_1}^{\ell_2} |u|^2 dx \leq \frac{(\ell_2-\ell_1)^2}{\pi^2} \int_{\ell_1}^{\ell_2} |u_x|^2 dx,
\]
a sufficient condition for $A_0 > 0$ is
\[
\frac{(1-\ell_1/\ell_2)^2}{(\ell_1/\ell_2)^2} < \frac{\pi^2}{\ell_2^2 + 1/2}.
\]
For fixed $\ell_2$ the above holds for $\ell_1/\ell_2$ close enough to $1$. 
\item The proof of Theorem \ref{theo:bessel} can be adapted to cover the case $A_0 \geq 0$. One removes the pole of $\mathbf{H}$ at $s = 0$ by considering $\dot{y}$ instead of $y$. 
\end{itemize}
\end{Remark}
\section{Finite dimensional systems}
In all this section we consider a finite dimensional LTI system $\Sigma$ described by \eqref{eq:equations_LTI} with \eqref{eq:dimensions}. We denote by $\mathbb{F}$ its input-output map. 
\subsection{Literature}\label{sec:literature}
Without loss of generality we keep assuming that $\opIm[C \mathfrak{C} | D] = \mathbb{C}^q$. In \cite[Theorem 1]{mesarovic} the authors study the \textit{functional reproducibility} of \eqref{eq:equations_LTI} with $D=0$, and claim to characterize the latter property by the surjectivity of a certain matrix built from $(A,B,C)$. With their definition of functional reproducibility, their Theorem 1 is wrong\footnote{From their Definition 2 and using our notations, if \eqref{eq:equations_LTI} is functionally reproducible then signals $\varphi \in C^d([0,\tau];\mathbb{C}^q)$ close enough in $C^d$ to an uncontrolled output of \eqref{eq:equations_LTI} (with $x(0) \in \mathbb{C}^d$ not necessarily zero and $u(t) \equiv 0$), satisfies $\varphi(0) \in \opIm CB$. Since their criterion of \cite[Theorem 1]{mesarovic} does not entail $CB$ to be surjective, this is a contradiction.}. Despite its importance, this work remains unclear\footnote{We do not understand the end of the proof \cite[Theorem 1]{mesarovic}, at the bottom of p. 558. It is not clear to us why (using their notations) $\| E \|_p < \infty$ entails that $\mathbf{X}'(t)$ solving (28) is free of impulse.} to us and we wish to clarify this. As we shall see, the essence of \cite[Theorem 1]{mesarovic} is true.

The subsequent work \cite{sain_massey} deals with the left-invertibility of LTI systems \eqref{eq:equations_LTI}, with general $D$. For $L \in \mathbb{N}$ we say that $\Sigma$ has an $L$-integral inverse if there is another (finite-dimensional) LTI system $\Sigma'$ whose transfer function $\mathbf{G}$ satisfies 
\[
\mathbf{G}(s)\mathbf{H}(s) = \frac{1}{s^L} I_p.
\]
If such a system $\Sigma'$ exists, the rational matrix $\mathbf{H}$ has the left-inverse $s^L\mathbf{G}(s)$, which encodes the action of the system $\Sigma'$ followed by $L$ successive time derivations. From the viewpoint of matrix algebra, it should be noticed that $s^L\mathbf{G}(s)$ is not any rational matrix. A close inspection of \cite{sain_massey} shows that such $L$ and $\mathbf{G}$ exist if and only if $\mathbf{H}$ is injective, and by transposition one obtains the criterion of \cite[Theorem 1]{mesarovic} when $D = 0$. The paper \cite{sain_massey} further gives an algorithm to compute the smallest $L$ (if exists) such that a system has an $L$-integral inverse, and show that in the positive case one may always take $L \leq d$. The work \cite{willsky} improves this to $L \leq d-\dim \opker D +1$.

The interested reader may further consult \cite{hautus} (see also \cite[\S 8.1-8.2]{trentelman}), which study the mapping $\mathbb{F}$ over the ring of impulsive-smooth distributions and show that $\mathbb{F}$ is surjective over that ring if and only if the transfer function is surjective (see \cite[Theorem 8.13]{trentelman}). This has the notable consequence that, provided the transfer function is surjective, any smooth signal can be realised as an output of the system, the caveat being that the control may contain impulsive terms, of the form $\delta^{(k)}(t)$. 

We also refer to \cite{tucsnak_rissel}, which studies the tracking problem for quadratic systems, and as a preliminary result shows that when $D = 0$ and $C = I_d$, the operator $\mathbb{F} : L^2(0,\tau;\mathbb{C}^p) \to L^2(0,\tau;\mathbb{C}^q)$ has dense range if and only if $B$ is surjective. 

Finally we mention that for SISO systems, \textit{i.e.} systems for which $p=q=1$, in controller form, the image of $\mathbb{F}$ has been characterized in \cite[Theorem 5.1]{zzz25}. 
\subsection{Some elementary facts}
In this section we show Propositions \ref{prop:exact_tracking_EDO} and \ref{prop:tracking_MISO}.
\begin{proof}[Proof of Proposition \ref{prop:exact_tracking_EDO}]
We start with the converse direction, assume that $D$ is surjective and let $R : \mathbb{C}^q \to \mathbb{C}^p$ be a right inverse of $D$. We parametrize the control law $u$ by $Rv$ with $v \in L^2(0,\tau; \mathbb{C}^q)$, the generated output $y$ writes 
\[
y(t) = \int_0^t Ce^{(t-\sigma)A}BRv(\sigma)d\sigma + v(t). 
\]
Let $\mathcal{V}$ be the above Volterra integral operator on $L^2(0,\tau; \mathbb{C}^q)$, it makes a compact operator. Thus, the operator $I + \mathcal{V}$ satisfies the Fredholm alternative, and it is injective if and only if surjective. Using the Banach fixed point theorem one shows that $I + \mathcal{V}$ is injective, hence surjective, whence the claim. 
\newline
\newline
For the direct implication we reason by contraposal and assume that $D$ is not surjective. In particular, $\opIm D$ is a strict subspace of $\mathbb{C}^q$, let $V$ be a linear supplement and $\pi : \mathbb{C}^q \to V$ be the projection onto $V$ parallel to $\opIm D$. For $u \in L^2(0,\tau;\mathbb{C}^p)$ we have 
\[
\pi y(t) = \int_0^t \pi Ce^{(t-\sigma)A}Bu(\sigma)d\sigma,
\]
which is of class $H^1_{(0)}(0,\tau;V)$. The signals $y$ such that $\pi y$ does not lie in the latter class cannot be reached, hence the claim. 
\end{proof}
\begin{proof}[Proof of Proposition \ref{prop:tracking_MISO}]
The case $D \neq 0$ is contained in Proposition \ref{prop:exact_tracking_EDO}. Let us then assume that $D = 0$ but $CB \neq 0$, the general case following by induction. For any $u \in L^2(0,\tau;\mathbb{C}^p)$ we have
\begin{align*}
\dot{y}(t) &= \frac{d}{dt}\int_0^t Ce^{(t-\sigma)A}Bu(\sigma)d\sigma \\
&= CBu(t) + \int_0^t CAe^{(t-\sigma)A}Bu(\sigma)d\sigma, 
\end{align*}
in $\mathcal{D}'(0,\tau)$. We deduce that $y \in H^1_{(0)}(0,\tau)$, and Proposition \ref{prop:exact_tracking_EDO} shows that $\mathcal{Y}(0,\tau) = H^1_{(0)}(0,\tau)$.
\end{proof}
The number $\nu$ is the order of the zero of $t \mapsto Ce^{tA}B$ at the origin, the larger this number is the smoother the output signals are. This reasoning can be generalized as follows: assume that $y(t)$ is given by 
\[
y(t) = \int_0^t k(t-\sigma)u(\sigma)d\sigma,
\]
with $k \in C^\infty_{(0)}([0,\tau] ; \mathcal{M}_{q,p}(\mathbb{C}))$. Then, $\mathcal{Y}(0,\tau) \subset C_{(0)}^\infty([0,\tau];\mathbb{C}^q)$. Of course, if the kernel $k$ further belongs to a Denjoy-Carleman (or ultra-differentiable) class over $[0,\tau]$, then one can deduce similar properties for the outputs. 

\subsection{Proof of the main result}\label{sec:app_L2}
In this \S \ we show Proposition \ref{prop:dense_range}, we begin by a result concerning the injectivity of $\mathbb{F}$.
\begin{Proposition}\label{prop:into_toeplitz_finite_dim}
The following assertions are equivalent to each others. 
\begin{enumerate}
\item $\mathbb{F}$ is injective $L^2(0,\infty; \mathbb{C}^p) \to L^2_{\oploc}([0,\infty);\mathbb{C}^q)$.
\item $\mathbf{H}$ is an injective rational matrix.
\item The matrix 
\[ N_D := 
\left( \begin{array}{cccc}
    D & 0 & \dots & 0 \\
    CB & D & \dots & 0 \\
    CAB & CB & \dots & 0 \\
    \vdots & \vdots & & \vdots \\
    CA^{d-1}B & CA^{d-2}B & \dots & D \\
    CA^dB & CA^{d-1}B & \dots & CB \\
    \vdots & \vdots & & \vdots \\
    CA^{2d-1}B & CA^{2d-2}B & \dots & CA^{d-1}B
\end{array} \right),
\]
is injective. 
    \item $\mathbf{H}$ has the left-inverse $s^L \mathbf{G}(s)$ for some $L \in \mathbb{N}$ and $\mathbf{G}$ the transfer function of a finite dimensional LTI system.
    \item There holds $p \leq d$ and the previous assertion holds with $L = \min(d-\dim \opker D+1,d)$.
\end{enumerate}
\end{Proposition}
Prior to give the proof we set up some notations. For $\sigma \in \mathbb{R}$ we let $L^2_\sigma(0,\infty)$ be the Hilbert space with norm 
\[
\| f \|^2 = \int_0^\infty |e^{-\sigma t} f(t)|^2 dt.
\]
We let $\gamma = \gamma(\mathbb{F}) \in [-\infty,\omega_0(e^{tA})]$ be such that 
\[
\forall \gamma < \sigma < \infty,\quad \forall u \in L^2(0,\infty;\mathbb{C}^p), \quad \int_0^\infty \| e^{-\sigma t} y(t) \|^2 dt < \infty,
\]
see \cite[Theorem 4.1]{weiss_transfer}. The operator $\mathbb{F}$ is thus bounded 
\[
    L^2(0,\infty;\mathbb{C}^p) \to L^2_{\sigma}(0,\infty;\mathbb{C}^q),\quad L^2_{\oploc}([0,\infty); \mathbb{C}^p) \to L^2_{\oploc}([0,\infty);\mathbb{C}^q),
\]
for all $\gamma < \sigma < \infty$. Moreover, for $u \in L^2(0,\infty;\mathbb{C}^p)$, the function $y$ is Laplace transformable with abscissa of absolute convergence $\leq \gamma$, and we have 
\begin{equation} \label{eq:transfer_Laplace}
    \hat{y}(s) = \mathbf{H}(s) \hat{u}(s),\quad \opRe s > \gamma.
\end{equation}
For $n \in \mathbb{N}$ and $\alpha \in \mathbb{R}$ denote $H^2(\mathbb{C}_\alpha ; \mathbb{C}^n)$ the set of functions $F = \mathbb{C}_\alpha \to \mathbb{C}^n$ such that each component lies in the Hardy space $H^2(\mathbb{C}_\alpha)$. Note that a function $f : \mathbb{C}_\alpha \to \mathbb{C}$ lies in $H^2(\mathbb{C}_\alpha)$ if and only if $f(\cdot - \alpha) \in H^2(\mathbb{C}_0)$, hence the Paley-Wiener theorem implies that the Laplace transform is an isomorphism $L^2_\alpha(0,\infty;\mathbb{C}^n) \to H^2(\mathbb{C}_\alpha ; \mathbb{C}^n)$. For the Paley-Wiener theorem and Hardy spaces we refer to \cite{rudin,koosis}.
\begin{proof}
Let us show that $1 \Longrightarrow 2$. Assume that $\mathbb{F}$ is injective $L^2(0,\infty; \mathbb{C}^p) \to L^2_{\oploc}([0,\infty);\mathbb{C}^q)$ and let $\gamma < \sigma < \infty$. Since $\mathbb{F}$ is bounded $L^2(0,\infty;\mathbb{C}^p) \to L^2_{\sigma}(0,\infty;\mathbb{C}^q)$, from \eqref{eq:transfer_Laplace} one sees that the associated Toeplitz operator 
\[
\mathscr{T} : H^2(\mathbb{C}_0;\mathbb{C}^p) \to H^2(\mathbb{C}_\sigma;\mathbb{C}^q),\quad \mathscr{T}f(s) = \mathbf{H}(s) f(s),
\]
is bounded and injective by the hypothesis. Let $f \in \mathbb{C}(s)^p$ be such that $\mathbf{H}f = 0$. Put 
\[
\tilde{f}(s) = \frac{p(s)}{(1+s)^n}f(s),\quad p \in \mathbb{C}[s],
\]
with $p \not \equiv 0$ removing the poles of $f$ and $n$ large enough so that $\tilde{f} \in H^2(\mathbb{C}_0;\mathbb{C}^p)$. Then
\[
0 = \frac{p(s)}{(1+s)^n} \mathbf{H}(s)f(s) = \mathbf{H}(s) \frac{p(s)}{(1+s)^n}f(s) = (\mathscr{T} \tilde{f})(s),
\]
hence $\tilde{f} \in \opker \mathscr{T}$. By assumption we deduce successively that $\tilde{f} \equiv 0$ and $f \equiv 0$. 
\newline
\newline
The converse implication $2 \Longrightarrow 1$ is shown in a similar way, we omit the details. The equivalence $2 \Longleftrightarrow 4$ is from \cite[Corollary 2]{sain_massey}. The equivalence $3 \Longleftrightarrow 4$ is from \cite[Theorem 3]{sain_massey}. The equivalence $4 \Longleftrightarrow 5$ is from \cite[Corollary 1]{willsky} and \cite[Corollary 1]{sain_massey}. 
\end{proof}
To prove Proposition \ref{prop:dense_range} we rely on a Lemma. We define for $L \in \mathbb{N}$ and $\alpha \in \mathbb{R}$ the set
\[
\dot{H}_{0,\alpha}^L(0,\infty) = \{ f \in H^L_{0,\oploc}[0,\infty) : f^{(L)} \in L^2_\alpha(0,\infty) \}.
\]
\begin{Lemma}\label{lem:inv_order} Let $\alpha >0$ and $L \in \mathbb{N}$. If $\mathbf{H}$ has a right inverse $\mathbf{G}$ free of pole on $\{ \opRe s \geq \alpha \}$ and such that 
\[
\| \mathbf{G}(s) \| \lesssim |s|^L,\quad |s| \to \infty,
\]
then we have the inclusion 
\begin{equation}\label{eq:incl_sobo}
    \dot{H}^L_{0,\alpha}(0,\infty;\mathbb{C}^q) \subset \mathbb{F} L^2_\alpha (0,\infty;\mathbb{C}^p).
\end{equation}
\end{Lemma}
\begin{proof}
For $y \in \dot{H}_{0,\alpha}^L(0,\infty)$ we consider the control law $u$ defined formally by 
\[
\hat{u}(s) = \mathbf{G}(s) \hat{y}(s).
\]
To conclude, from the Paley-Wiener theorem it is enough to show that the right hand side is in $H^2(\mathbb{C}_\alpha;\mathbb{C}^p)$. The function $\mathbf{G}\hat{y}$ is holomorphic $\mathbb{C}_\alpha  \to \mathbb{C}^p$ and 
\[
\mathbf{G}(s) \hat{y}(s) = \mathbf{G}(s) \frac{1}{s^L} \widehat{y^{(L)}}(s) = \frac{1}{s^L}\mathbf{G}(s) \widehat{y^{(L)}}(s),\quad \opRe s > 0. 
\]
As $\alpha > 0$ the rational matrix $\mathbf{G}(s)/s^L$ is bounded on $\mathbb{C}_\alpha$, hence the claim. 
\end{proof}
\begin{proof}[Proof of Proposition \ref{prop:dense_range}]
The implications $1 \Longleftrightarrow 2 \Longrightarrow 3$ and $8 \Longrightarrow 7 \Longrightarrow 2$ are trivial. We begin by showing that $3 \Longrightarrow 2$, so that in particular $1-3$ are equivalent. Let $0 < \tau < \tau' < \infty$ and assume that $\mathbb{F} : L^2(0,\tau;\mathbb{C}^p) \to L^2(0,\tau;\mathbb{C}^q)$ has dense range. Let $f \in L^2(0,\tau';\mathbb{C}^q)$ and $\epsilon > 0$, by assumption there exists $u \in L^2(0,\tau;\mathbb{C}^p)$ such that 
\[
\| \mathbb{F}u - f \|_{L^2(0,\tau;\mathbb{C}^q)} < \epsilon. 
\]
We extend $u(t)$ by $0$ for times $\tau < t < \tau'$ and assume to fix the ideas that $\tau' \leq 2 \tau$. Using again the denseness assumption one finds $v \in L^2(0,\tau'-\tau;\mathbb{C}^p)$ such that 
\[
\| \mathbb{F}v + \mathbb{L} \Phi_\tau u - f(\tau+\cdot) \|_{L^2(0,\tau'-\tau;\mathbb{C}^q)} < \epsilon,
\]
where we have introduced two bounded operators 
\[
\Phi_\tau u = \int_0^\tau e^{t-\sigma}Bu(\sigma) d\sigma,\quad L^2(0,\tau;\mathbb{C}^p) \to \mathbb{C}^d,
\]
and 
\[
(\mathbb{L} x)(t) = Ce^{tA}x,\quad \mathbb{C}^d \to L^2_{\oploc}([0,\infty);\mathbb{C}^q). 
\]
The control 
\[
w(t) = \left\lbrace \begin{array}{rl}
    u(t), & 0 < t < \tau,  \\
    v(t-\tau), & \tau < t < \tau', 
\end{array}\right.
\]
then satisfies 
\[
    \| \mathbb{F}w - f \|_{L^2(0,\tau';\mathbb{C}^q)} = \| \mathbb{F}w - f \|_{L^2(0,\tau;\mathbb{C}^q)} + \| \mathbb{F}w - f \|_{L^2(\tau,\tau';\mathbb{C}^q)} 
\]
where 
\[
 \| \mathbb{F}w - f \|_{L^2(0,\tau;\mathbb{C}^q)} = \| \mathbb{F}u - f \|_{L^2(0,\tau;\mathbb{C}^q)} < \epsilon,
\]
and 
\[
(\mathbb{F}w)(t + \tau) = (\mathbb{F}v)(t) + (\mathbb{L} \Phi_\tau u)(t),
\]
see \cite[eq. (1.6)]{weiss_representation}. We deduce that $2$ holds, under the additional assumption $\tau' \leq 2 \tau$. For general $\tau' > \tau$ the result follows by induction. 
\newline
\newline
We show that $3 \Longrightarrow 4$. Assume $3$ and let $0 < \tau < \infty$ be such that $\mathbb{F} : L^2(0,\tau;\mathbb{C}^p) \to L^2(0,\tau;\mathbb{C}^q)$ has dense range. We show that $\mathbf{H}$ is a surjective rational matrix. For this it suffices to show that 
\[
\forall g \in \mathbb{C}(s)^q,\quad g(s)^T\mathbf{H}(s) \equiv 0 \Longrightarrow g(s) \equiv 0.
\]
In fact, it is enough to show the above assertion for $g \in \mathbb{C}[s]^q$. To show this, let $g$ be as such and consider the differential operator 
\[
L : \mathcal{D}'(\mathbb{R};\mathbb{C}^p) \to \mathcal{D}'(\mathbb{R}),\quad L = g(\partial_t)^T,
\]
as well as the convolution kernel 
\[
k(t) = Ce^{tA}B1_{t > 0} + \delta_0(t) D \in \mathcal{D}'(\mathbb{R};\mathcal{M}_{q,p}(\mathbb{C})).
\]
For $u \in C^\infty_c((0,\tau) ; \mathbb{C}^p)$ we have $\mathbb{F}u = k*u$, hence 
\[
L \mathbb{F} u = L(k*u) = (Lk)*u.
\]
The above distribution is Laplace transformable, with abscissa of absolute convergence $\leq \gamma$, hence 
\[
\mathcal{L}(L \mathbb{F} u )(s) = (\mathcal{L}Lk)(s) \hat{u}(s) = g(s)^T\hat{u}(s) = 0 , \quad \opRe s > \gamma.
\]
By uniqueness of the Laplace transform, $L \mathbb{F}u = 0$, and this for all $u \in C^\infty_c((0,\tau) ; \mathbb{C}^p)$. Now for all $u \in C^\infty_c((0,\tau) ; \mathbb{C}^p)$ and $\varphi \in C^\infty_c((0,\tau) ; \mathbb{C}^q)$ we obtain 
\[
0 = \int_0^\tau \langle L \mathbb{F}u , \varphi(t) \rangle_{\mathbb{C}^q}dt = - \int_0^\tau \langle \mathbb{F}u , L^*\varphi(t) \rangle_{\mathbb{C}^q}dt,
\]
with $L^*$ the formal adjoint of $L$. By denseness of the range of $\mathbb{F} : L^2(0,\tau;\mathbb{C}^p) \to L^2(0,\tau;\mathbb{C}^q)$ and the denseness of $C^\infty_c((0,\tau) ; \mathbb{C}^p) \subset \mathbb{F} : L^2(0,\tau;\mathbb{C}^p)$, we deduce that 
\[
\forall \varphi \in C^\infty_c((0,\tau) ; \mathbb{C}^q),\quad L^* \varphi = 0.
\]
Since $L^*$ is a differential operator with constant coefficients, we deduce that $L^* = 0$, hence $L = 0$. We conclude that $g = 0$, hence the claim. 
\newline
\newline
The implication $4 \Longrightarrow 8$ follows from Lemma \ref{lem:inv_order} and Proposition \ref{prop:into_toeplitz_finite_dim} by transposition. The equivalence $4 \Longleftrightarrow 5$ is classical result of linear algebra, see \cite[Exercise 5.48]{abadir}. The equivalences $4 \Longleftrightarrow 6$ follows from Proposition \ref{prop:into_toeplitz_finite_dim} by transposition, which ends the proof. 
\end{proof}
\section{SISO systems}\label{sec:app_SISO}
In this section we prove Theorem \ref{theo:charac_output_cartwright}. As already discussed in the introduction, this result has already been proved in \cite{davron_lissy} using \cite[Lemma 3.4]{foures}, which provides a generalization of the Paley-Wiener Theorem. To enlighten the proof of Theorem \ref{theo:charac_output_cartwright} we begin by recalling and improving the Paley-Wiener Theorem.
\subsection{Paley-Wiener Theorem}
The following result is due to R. Paley and N. Wiener \cite{paley_wiener}. Subsequently, it has been customary to call ``Paley-Wiener Theorem" any result which links the growth of a function and the complex analytic properties its Fourier transform. Most of the references on the topic work with the Fourier transform and holomorphic functions on the upper half plane, to be close to them we work on the upper half plane $\Pi := \{ x+iy : y > 0 \}$ and define the Fourier transform of a function $f : \mathbb{R} \to \mathbb{C}$ (provided this makes sense) as 
\[
\mathcal{F}f(\xi) = \frac{1}{2 \pi} \int_\mathbb{R} e^{ix\xi}f(x)dx,
\]
which is not essential. We identify the set of $L^2(\mathbb{R})$ functions supported on $[0,\infty)$ with $L^2(0,\infty)$. 
\begin{Theorem}{\cite[Theorem 19.2]{rudin}}
The Fourier transform is an isometric isomorphism $L^2(0,\infty) \to H^2(\Pi)$. 
\end{Theorem}
For the theory of Hardy spaces see \cite{rudin,koosis}. Recall that $H^2(\Pi)$ is defined as the set of these holomorphic functions $F : \Pi \to \mathbb{C}$ satisfying the growth bound
\begin{equation}\label{eq:def_hardy}
    \| F \|_{H^2(\Pi)}^2 := \sup_{y > 0} \int_\mathbb{R} |F(x+iy)|^2 dx < \infty.
\end{equation}
Such an $F(x+iy)$ has an $L^2$ boundary value as $y \to 0^+$, which we abusively denote $F(x)$, \textit{viz.}
\[
F(x+iy) \xrightarrow[y \to 0^+]{L^2(\mathbb{R},dx)} F(x),
\]
see \textit{e.g.} \cite[Theorem 7.2]{katznelson}. Moreover, we have $f = \mathcal{F}^{-1}[F(x)]$, from which one deduces
\begin{equation}\label{eq:inf_spectrum}
    \inf \opsupp f = - \limsup_{y \to \infty} \frac{1}{y} \log | F(iy)|,
\end{equation}
see the proof of \cite[Theorem, p. 179]{koosis}. 
For $F \in H^2(\Pi)$, we have in particular
\begin{equation}\label{eq:spec_0}
    \limsup_{y \to \infty} \frac{1}{y} \log | F(iy)| \leq 0.
\end{equation}
To go further we introduce a class of holomorphic functions on $\Pi$. 
\begin{Definition}
A holomorphic function $F : \Pi \to \mathbb{C}$ is said to belong to the Cartwright class $\mathscr{C}(\Pi)$ if it is continuous up to the real axis, satisfies 
\[
\exists C > 0,\quad \forall z \in \Pi,\quad |F(z)| \leq Ce^{C|z|},
\]
and 
\begin{equation}\label{eq:log_int}
    \int_\mathbb{R} \frac{\log^+ |F(x)|}{1+x^2}dx < \infty. 
\end{equation}
\end{Definition}
\begin{Remark}\label{rem:cart_hardy}
Essentially, the Cartwright class is larger than $H^2(\Pi)$. This is not true because $H^2(\Pi)$ functions do not necessarily possess a continuous boundary value, but this can be circumvented replacing $F \in H^2(\Pi)$ by $F(x + iy + i\epsilon)$ for $\epsilon > 0$, which turns out to belong to $\mathscr{C}(\Pi)$ (see \cite[Lemma, p. 149]{koosis}). 
\end{Remark}
Functions of the Cartwright class bear many useful properties, the first of which being 
\begin{equation}\label{eq:growth_cart}
    \log | F(z) | \leq Ay + \frac{y}{\pi} \int_\mathbb{R}  \frac{\log |F(t)|}{|z-t|^2}dt,\quad \forall z = x+iy \in \Pi,\quad A := \limsup_{y \to \infty} \frac{1}{y} \log | F(iy)|,
\end{equation}
see \cite[Theorem, \S III.G.2, p. 51]{K1}.
\newline
\newline
Observe that in \eqref{eq:inf_spectrum}, for the left term to be defined one has to make sense of $f = \mathcal{F}^{-1}[F(x)]$. This is tedious when $F \in \mathscr{C}(\Pi)$, as the boundary value satisfies \eqref{eq:log_int}, and in general not better. However, the right term of \eqref{eq:inf_spectrum} is well-defined when $F \in \mathscr{C}(\Pi)$. Below is the claimed generalization of the Paley-Wiener Theorem. 
\begin{Lemma}\label{lem:pwc}
Let $F \in \mathscr{C}(\Pi)$. Then $F \in H^2(\Pi)$ if and only if it satisfies both \eqref{eq:spec_0} and $F \in L^2(\mathbb{R})$. 
\end{Lemma}
\begin{Remark}The reader familiar with harmonic analysis on the unit disc will notice that the above result is analogous to the Smirnov Theorem (see \textit{e.g.} \cite[Theorem 2.11]{duren} or \cite[\S II.5]{garnett}). 
\end{Remark}
\begin{proof}
The direct implication is trivial given the above discussion. For the converse one let $F \in \mathscr{C}(\Pi)$ satisfy both \eqref{eq:spec_0} and $F \in L^2(\mathbb{R})$. From \eqref{eq:growth_cart} we deduce that for all $z = x + iy \in \Pi$,
\[
\log | F(z)| \leq \frac{y}{\pi} \int_\mathbb{R}  \frac{\log |F(t)|}{|z-t|^2}dt.
\]
One notices that 
\begin{equation}\label{eq:prob_Poisson}
    \frac{y}{\pi} \frac{dt}{|z-t|^2},
\end{equation}
is a probability measure on $\mathbb{R}$, for every $z = x + iy \in \Pi$. From the Jensen inequality applied to $\lambda \mapsto e^{2 \lambda}$ we deduce 
\[
|F(z)|^2 \leq \frac{y}{\pi} \int_\mathbb{R}  \frac{|F(t)|^2}{|z-t|^2}dt. 
\]
We conclude with the inequality 
\[
\| u * v \|_{L^2(\mathbb{R})} \leq \|u\|_{L^2(\mathbb{R})} \|v\|_{L^1(\mathbb{R})},
\]
using that \eqref{eq:prob_Poisson} is integrable and the hypothesis $F \in L^2(\mathbb{R})$.
\end{proof}
We end this \S \ by stating two properties of the Cartwright class. 
\begin{Proposition}\label{prop:cart_inv}
Let $F \in \mathscr{C}(\Pi)$ never vanishing in $\Pi \cup \mathbb{R}$. Then $1/F \in \mathscr{C}(\Pi)$. 
\end{Proposition}
\begin{Theorem}{\cite[Remark, p. 118, \S 16.1]{levin_lectures}}\label{theo:krein}
Let $F,G \in \mathscr{C}(\Pi)$, then 
\begin{equation}\label{eq:krein}
    \limsup_{y \to \infty} \frac{1}{y} \log |F(iy)G(iy)| = \limsup_{y \to \infty} \frac{1}{y} \log |F(iy)| + \limsup_{y \to \infty} \frac{1}{y} \log |G(iy)|.
\end{equation}
\end{Theorem}
Both Theorem \ref{theo:krein} and Proposition \ref{prop:cart_inv} are contained in \cite[\S 16.1]{levin_lectures}. The proof of Theorem \ref{theo:krein} is done in a less general situation but the quoted reference explains how to overcome this. Proposition \ref{prop:cart_inv} is not explicitly stated but can be shown by a straightforward adaptation of the arguments developed therein. 
\subsection{Proof of the Theorem}
We are now in position to show Theorem \ref{theo:charac_output_cartwright}. Let $\mathbf{H}$ be as in the statement and fix $y \in \mathcal{Y}(0,\infty)$. Let $u \in L^2(0,\infty)$ be such that 
\[
\hat{y}(s) = \mathbf{H}(s) \hat{u}(s),\quad \opRe s > 0.
\]
Since $\mathbf{H}$ never vanishes on $\mathbb{C}_0$, the quotient $\hat{y} / \mathbf{H}$ is holomorphic $\mathbb{C}_0 \to \mathbb{C}$ and the above relation entails $\hat{y} / \mathbf{H} \in H^2(\mathbb{C}_0)$. In particular, the latter function has a boundary value which is square integrable. Since $\mathbf{H}$ is continuous and never vanishes on $\mathbb{C}_0 \cup i \mathbb{R}$, one sees that $\hat{y} / \mathbf{H}$ has the boundary value $\hat{y}(i\tau) / \mathbf{H}(i\tau)$ in $L^1_{\oploc}(\mathbb{R})$. Since the $L^2$ convergence is stronger than the $L^1_{\oploc}$ convergence and the latter is Hausdorff, we have $\hat{y}(i\tau) / \mathbf{H}(i\tau) \in L^2(\mathbb{R})$, hence the first assertion in \eqref{eq:charac_output}. For the second one we reason as follows. From Proposition \ref{prop:cart_inv} we have $1/\mathbf{H} \in \mathscr{C}(\mathbb{C}_0)$, where $\mathscr{C}(\mathbb{C}_0)$ is obtained from $\mathscr{C}(\Pi)$ by a rotation of the argument. From Remark \ref{rem:cart_hardy} there holds $\hat{y}(1+\cdot) \in \mathscr{C}(\mathbb{C}_0)$. Obviously $1/\mathbf{H}(\cdot + 1) \in \mathscr{C}(\mathbb{C}_0)$ and $(\hat{y}/\mathbf{H})(\cdot + 1) \in H^2(\mathbb{C}_0)$. Therefore, owing to Theorem \ref{theo:krein} and \eqref{eq:spec_0}, we have 
\[
0 \geq \limsup_{\sigma \to \infty} \frac{1}{\sigma} \log \left| \frac{\hat{y}(\sigma+1)}{\mathbf{H}(\sigma+1)}\right| = \limsup_{\sigma \to \infty} \frac{1}{\sigma} \log | \hat{y}(\sigma)| + \limsup_{\sigma \to \infty} \frac{1}{\sigma} \log \left| \frac{1}{\mathbf{H}(\sigma)}\right|.
\]
From 
\begin{equation}\label{eq:inf_supp_laplace}
    \inf \opsupp y = - \limsup_{\sigma \to \infty} \frac{1}{\sigma} \log | \hat{y}(\sigma)|,
\end{equation}
(recall \eqref{eq:spec_0}) the second assertion of \eqref{eq:charac_output} follows. 
\newline
\newline
The other direction of the proof can be done either by applying \cite[Lemma 3.4]{foures} or by reasoning as follows. Let $y \in L^2(0,\infty)$ be such that \eqref{eq:charac_output} holds. The function $1/\mathbf{H}$ belongs to $\mathscr{C}(\mathbb{C}_0)$, hence from \eqref{eq:growth_cart} it satisfies 
\[
\log | 1/\mathbf{H}(s) | \leq \delta \sigma + \frac{\sigma}{\pi} \int_\mathbb{R}  \frac{\log |1/\mathbf{H}(it)|}{|s-it|^2}dt,\quad s = \sigma + i\tau \in \mathbb{C}_0,\quad \delta := \limsup_{y \to \infty} \frac{1}{\sigma} \log \left| \frac{1}{\mathbf{H}(\sigma)}\right|.
\]
The similar estimate 
\[
\log | \hat{y}(s) | \leq \mu \sigma + \frac{\sigma}{\pi} \int_\mathbb{R}  \frac{\log |\hat{y}(it)|}{|s-it|^2}dt,\quad s = \sigma + i\tau \in \mathbb{C}_0, \quad \mu := \limsup_{y \to \infty} \frac{1}{\sigma} \log| \hat{y}(\sigma)|.
\]
holds for $\hat{y}$, since $e^{-\mu s}\hat{y}(s) \in  H^2(\mathbb{C}_0)$, see the discussion after \cite[Theorem 4.1]{garnett}. Summing these two estimates and using $\hat{y}(it) = \mathcal{F}y(t)$ we find
\[
\log | \hat{y}(s)/\mathbf{H}(s) | \leq (\mu+\delta) + \frac{\sigma}{\pi} \int_\mathbb{R}  \frac{\log |\mathcal{F}y(t)/\mathbf{H}(it)|}{|s-it|^2}dt.
\]
Reasoning as in the proof of Lemma \ref{lem:pwc}, one sees that $\hat{y}/\mathbf{H} \in H^2(\mathbb{C}_0)$ as soon as $\mu + \delta \leq 0$. From \eqref{eq:inf_supp_laplace} we are done. 
\subsection{An application}
We are not really interested in the well-posedness of \eqref{eq:bessel} as a LTI system, but rather in studying its input-output map. For $u \in H^1_0(0,\infty)$ the change of variable 
\[
\xi(t,x) := z(t,x) - \frac{x-\ell_2}{\ell_1-\ell_2}u(t),
\]
transforms \eqref{eq:bessel} into
\begin{equation}\label{eq:bessel2}
    \left\lbrace \begin{array}{rcl c c}
        \xi_t(t,x) &=& x^2 \xi_{xx}(t,x) + x\xi_x(t,x) + x^2 \xi(t,x) + f(t,x), &t > 0,&  \ell_1 < x < \ell_2,  \\
        \xi(t,\ell_1) & = & 0,& t > 0,\\
        \xi(t,\ell_2) &=& 0,& t > 0, \\
        \xi(0,x) &=& 0,&& \ell_1 < x < \ell_2,
    \end{array}\right.
\end{equation}
with 
\[
f(t,x) := - \frac{x-\ell_2}{\ell_1-\ell_2}\dot{u}(t) - x \frac{\ell_2}{\ell_1-\ell_2}u(t) + x^2 \frac{x-\ell_2}{\ell_1-\ell_2}u(t).
\]
By assumption, the operator $A_0$ is self-adjoint and strictly positive on the state space $X := L^2(\ell_1,\ell_2)$. By maximal regularity (see \cite[Proposition 3.7]{bensoussan_control}), the solution $\xi$ of \eqref{eq:bessel2} satisfies 
\[
\xi \in H^1(0,\infty;X) \cap L^2(0,\infty;D(A_0)).
\]
For $u \in H^1_0(0,\infty)$ we define the solution $z$ of \eqref{eq:bessel} as
\[
z(t,x) := \xi(t,x) + \frac{x-\ell_2}{\ell_1-\ell_2}u(t),
\]
where $\xi$ is the solution of \eqref{eq:bessel2}. In particular, we have $z \in L^2(0,\infty;H^2(\ell_1,\ell_2))$ and the trace $y(t) := z_x(t,\ell_2)$ makes sense as a pointwise evaluation, and lies in $L^2(0,\infty)$. This defines a shift invariant bounded operator $\mathbb{F} : H^1_0(0,\infty) \to L^2(0,\infty)$. For $u \in H^1_0(0,\infty)$ we may pass \eqref{eq:bessel} (or equivalently \eqref{eq:bessel2}) to the Laplace transform with respect to time, to find, at least formally
\[
    \hat{y}(s) = \mathbf{H}(s) \hat{u}(s),\quad 
\mathbf{H}(s) := \frac{-Y_{\sqrt{s}}(\ell_2)J_{\sqrt{s}}'(\ell_2) + J_{\sqrt{s}}(\ell_2)Y_{\sqrt{s}}'(\ell_2)}{J_{\sqrt{s}}(\ell_2)Y_{\sqrt{s}}(\ell_1)-Y_{\sqrt{s}}(\ell_2)J_{\sqrt{s}}(\ell_1)},
\]
where $J_\nu$ (resp. $Y_\nu$) is the principal branch of the Bessel function of first (resp. second) kind, and $\sqrt{s}$ stands for the principal determination of the square root. Introduce the cross-product 
\[
p_\nu(w_1,w_2) := J_\nu(w_2)Y_\nu(w_1) - Y_\nu(w_2)J_\nu(w_1),\quad w_1,w_2 \in \mathbb{C} \setminus (-\infty,0],\quad \nu \in \mathbb{C},
\]
so that the denominator of $\mathbf{H}(s)$ is $p_{\sqrt{s}}(\ell_1,\ell_2)$. The numerator of $\mathbf{H}$ is the Wronskian of $J_{\sqrt{s}}$ and $Y_{\sqrt{s}}$, hence \cite[eq. 10.5.2]{nist}
\[
\mathbf{H}(s) = \frac{2}{\pi\ell_2} \frac{1}{p_{\sqrt{s}}(\ell_1,\ell_2)},
\]
still formally for the moment. The quantity $p_\nu(\ell_1,\ell_2)$ is an entire function of $\nu$ which has the following properties: 
\begin{itemize}
    \item For all $\lambda \in \mathbb{C}$, the number $\lambda$ is an eigenvector of $A_0$ if and only if 
    \[
    \lambda = -\nu^2,\quad p_\nu(\ell_1,\ell_2) = 0,
    \]
    for some $\nu \in \mathbb{C}$.
    \item The function $p_\nu(\ell_1,\ell_2)$ has the representation formula \cite{nicholson}
    \[
    p_\nu(\ell_1,\ell_2) = \frac{2}{\pi}\int_0^L f(t) \cosh(\nu t) dt,\quad L := \log \frac{\ell_2}{\ell_1},\quad f(t) := J_0\left(\sqrt{\ell_1^2 + \ell_2^2-2\ell_1\ell_2\cosh t}\right).
    \]
\end{itemize}
From the assumption $A_0 > 0$ we deduce that $p_\nu$ has zeros only on $(-\infty,-\epsilon]$ for some $\epsilon > 0$. Moreover, $s \mapsto p_{\sqrt{s}}(\ell_1,\ell_2)$ is holomorphic on $\mathbb{C}\setminus (-\infty,0]$, continuous on $\mathbb{C}_0 \cup i \mathbb{R}$, never vanishes therein, and has exponential order $1/2$ on $\mathbb{C}_0$. Thus, the formal computation yielding $\mathbf{H}$ is valid. Given Corollary \ref{coro:plancherel_gevrey} and Proposition \ref{prop:cart_inv} it is enough to show that 
\begin{equation}\label{eq:esti_H}
    \mathbf{H} \in H^\infty(\mathbb{C}_0),\quad \mathbf{H}(i\xi)^{-1} \asymp  e^{R |\xi|^{1/s}}, \quad \xi \in \mathbb{R},
\end{equation}
with 
\[
R := \frac{1}{\sqrt{2}} \log \frac{\ell_1}{\ell_2},\quad s := 2.
\]
Indeed, since $s \mapsto p_{\sqrt{s}}(\ell_1,\ell_2)$ has order $1/2$ on $\mathbb{C}_0$ the delay $\delta$ is automatically zero. In fact, it is enough to show the estimate in \eqref{eq:esti_H}. Indeed, with such estimate one deduces that $\mathbf{H}$ is bounded on $i \mathbb{R}$, and from \eqref{eq:growth_cart} one deduces that $\mathbf{H}$ is bounded on $\mathbb{C}_0$. 

To show the estimate, since both terms are continuous and never vanishing on $\mathbb{R}$, and their modulus is an even function of $\xi$, it is enough to show that 
\[
\mathbf{H}(i\xi)^{-1} \asymp e^{R |\xi|^{1/s}}, \quad \xi \to \infty.
\]
We compute with $\nu = \sqrt{i\xi}$ and $\xi > 0$ fixed,
\[
\mathbf{H}(i\xi)^{-1} \asymp p_\nu(\ell_1,\ell_2) \asymp \int_0^L f(t) \cosh(\nu t) dt.
\]
For large $\opRe \nu$ the term $\cosh(\nu t)$ looks like $e^{\nu t}/2$, which assumes its greatest magnitude at $t=L$. Using the continuity of $f$ on $[0,L]$ and $f(L) = J_0(0) \neq 0$, one indeed shows that 
\[
\int_0^L f(t) \cosh(\nu t) dt \sim \frac{1}{2}f(L) e^{\nu L},\quad \opRe \nu \to \infty. 
\]
This ends the proof. 
\printbibliography[heading = bibintoc]
\end{document}